\documentclass[11pt]{article}

\usepackage{amsthm,amsmath,amssymb,arydshln}

\newtheorem{theorem}{Theorem}
\newtheorem{lemma}[theorem]{Lemma}

\newtheorem{corollary}[theorem]{Corollary}

\renewcommand{\geq}{\geqslant}
\renewcommand{\leq}{\leqslant}
\renewcommand{\ge}{\geqslant}
\renewcommand{\le}{\leqslant}

\def\eref#1{$(\ref{#1})$}

\def\lref#1{Lemma~$\ref{#1}$}
\def\tref#1{Theorem~$\ref{#1}$}
\def\cref#1{Corollary~$\ref{#1}$}

\title{Overlapping latin subsquares and full products}

\author{
Joshua M. Browning\\
\small School of Mathematical Sciences\\[-0.75ex]
\small Monash University\\[-0.75ex]
\small Vic 3800, Australia\\
\small{\tt joshua.browning@sci.monash.edu.au}\\
\\
Petr Vojt\v{e}chovsk\'y\\
\thanks{Research supported by Enhanced Sabbatical grant of the University of Denver.}
\small Department of Mathematics\\[-0.75ex]
\small University of Denver\\[-0.75ex]
\small Denver, Colorado 80208, U.S.A.\\
\small{\tt petr@math.du.edu}\\
\\
Ian M.\ Wanless\\
\thanks{Research supported by ARC grants DP0662946 and DP1093320.}
\small School of Mathematical Sciences\\[-0.75ex]
\small Monash University\\[-0.75ex]
\small Vic 3800, Australia\\
\small{\tt ian.wanless@sci.monash.edu.au}
}

\date{}

\begin{document}

\maketitle

\begin{abstract}
We derive necessary and sufficient conditions for there to exist a latin square
of order $n$ containing two subsquares of order $a$ and $b$ that
intersect in a subsquare of order $c$. We also solve the case
of two disjoint subsquares. We use these results to show that:
\begin{itemize}
\item[(a)] A latin square of order $n$ cannot have more than $\frac nm{n
    \choose h}/{m\choose h}$ subsquares of order $m$, where
    $h=\lceil(m+1)/2\rceil$. Indeed, the number of subsquares of order $m$
    is bounded by a polynomial of degree at most $\sqrt{2m}+2$ in $n$.

\item[(b)] For all $n\ge5$ there exists a loop of order $n$ in which every
    element can be obtained as a product of all $n$ elements in some order
    and with some bracketing.
\end{itemize}
\end{abstract}

\section{Overlapping latin subsquares}

A $k\times n$ \emph{latin rectangle\/} is a $k\times n$ matrix
containing $n$ different symbols, with each symbol occurring exactly
once in each row and at most once in each column. If $k=n$ the latin
rectangle is a \emph{latin square\/}.  A \emph{subsquare} in a latin
square $L$ is a submatrix of $L$ that is a latin square in its own
right. The cells in a subsquare are not required to be contiguous.

%A \emph{latin rectangle} is a rectangular array such that none of its rows or
%columns contains a repeated symbol. A \emph{latin square} of order $n$ is an
%$n\times n$ latin rectangle.

%A \emph{submatrix} of a latin square $L$ is any set $R\times C$, where $R$ is a
%subset of rows of $L$ and $C$ is a subset of columns of $L$. 

%We begin by finding conditions for the existence of a latin square with
%overlapping subsquares of given orders.

It is well-known that if two subsquares of a latin square intersect
then their intersection is itself a subsquare. Also, a
subsquare of a latin square is either the whole square or it has at
most half the order of the whole square. Another fact that we will
frequently use is the following result due to Ryser~\cite{ryser}:

\begin{theorem}\label{t:ryser}
Suppose that $R$ is an $r\times s$ matrix with symbols from
$\{1,\dots,n\}$ such that no symbol occurs more than once in any row
or column. For $1\le i\le n$, let $\Gamma(i)$ be the number of occurrences
of $i$ in $R$. Then $R$ can be embedded into a latin square of order
$n$ if and only if $\Gamma(i)\ge r+s-n$ for every $i\in\{1,\dots,n\}$.
\end{theorem}

The goal of this first section is to find conditions under which a latin
square may have subsquares of two specified orders. We begin by treating
the case when the two subsquares overlap. The simplest way for this to
happen is for one subsquare to contain the other (for this to be possible
it is necessary and sufficient that the larger subsquare is at least
twice the order of the smaller one). A more interesting case is when
the subsquares intersect, but neither is inside the other:

\begin{theorem}\label{t:overlapsubsq}
Suppose $0<c<a\le b<n$ are integers. In order for there to exist a latin square
of order $n$ containing two subsquares of order $a$ and $b$ whose intersection 
is a subsquare of order $c$, it is necessary and sufficient that
\begin{equation}\label{e:easycond}
n-2b\ge a-2c\ge0
\end{equation}
and
\begin{equation}\label{e:hardcond}
(n-2a)(n-2b)\ge c^2-(n-2a-2b+3c)^2.
\end{equation}
\end{theorem}

\begin{proof}
To prove necessity, assume that $L$ is a latin square of the desired type. By
permuting rows and columns if necessary, we may assume that $L$ has the form
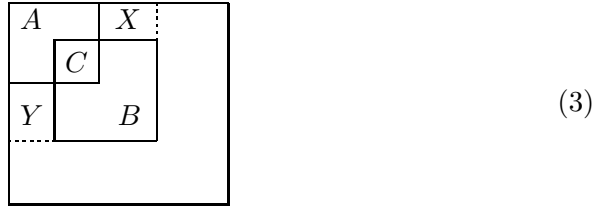
\begin{equation}\label{e:formofL}
\begin{tabular}{|c|c|c;{1pt/1.5pt}c|}
\hline
\multicolumn{2}{|l|}{$\!A\!$}&$X$&\\
\cline{2-3}
&$\!C\!$\vrule height 12pt width 0pt &\multicolumn{1}{c|}{}&\\
\cline{1-2}
$\!Y\!$&\multicolumn{2}{r|}{$B$\vrule height 16pt depth 6pt width 0pt}
&\vrule height 0pt depth 0pt width 15pt\\
\cdashline{1-1}[1pt/1.5pt]\cline{2-3}
\multicolumn{4}{|c|}{\vrule height 20pt width 0pt}\\
\hline
\end{tabular}
\end{equation}
where $A$ is a subsquare of order $a$, $B$ is a subsquare of order $b$ and they
intersect in $C$, a subsquare of order $c$.

%[had] Since $A\setminus C$ and $B\setminus C$ have no symbol in common, 
%This is not true unless you interpret A,B,C to be sets of symbols.
%If you interpret them as sets of entries (i.e. row,col,sym triples),
%which is most natural to me, then it is false.

Let $S$ be the set of symbols of $L$ that do not occur in $A$ or $B$. Note 
that $|S|=n-(a+b-c)$. The regions $X$ and $Y$ are $(a-c)\times(b-c)$ and
$(b-c)\times(a-c)$ submatrices, respectively, and they contain only symbols
from $S$. To have enough symbols to fill the first row of $X$ we must have
$|S|\ge b-c$, which gives the first inequality in \eref{e:easycond}. The
inequality $a\ge2c$ is immediate, given that $C$ is a subsquare of $A$ and
$c<a$.

To prove \eref{e:hardcond} we apply Ryser's condition to $R=A\cup
B\cup X\cup Y$ and conclude that each symbol in $S$ must occur at least
$2(a+b-c)-n$ times in $X\cup Y$. To fit this many symbols of $S$
into $X\cup Y$, we require
\begin{equation}\label{e:eqhardcond}
2(a-c)(b-c)=|X\cup Y|\ge \big(2(a+b-c)-n\big)|S|.
\end{equation}
Upon multiplying this inequality by $2$ and rearranging, 
it becomes \eref{e:hardcond}.

It remains to show sufficiency. Assuming \eref{e:easycond}, we have $b\ge
a\ge2c$ and so it is possible to construct $A$, $B$ and $C$ as in
\eref{e:formofL}. We aim to fill in $X\cup Y$ in such a way that the symbols in
$S$ each occur at least $\big\lfloor2(a-c)(b-c)/|S|\big\rfloor$ times in $X\cup
Y$. By \eref{e:eqhardcond} this would mean that each symbol in $S$ occurs at
least $\lfloor2(a+b-c)-n\rfloor=2(a+b-c)-n$ times in $R$. Furthermore, every
symbol of $A\cup B$ occurs at least $a$ times in $R$, and
$a\ge2(a+b-c)-n$ by \eref{e:easycond}. Therefore, if $X\cup Y$ can be filled as
desired, Ryser's condition holds for $R$, and $L$ exists.

To fill $X\cup Y$, we order the cells of $X\cup Y$ in a sequence such that any
subsequence of $2(a-c)$ consecutive terms contains cells from distinct rows and
columns. This is easily achieved by alternating ``diagonals'' of $X$ and
``diagonals'' of $Y$. We then fill the cells in the order determined by the
sequence, using all occurrences of one symbol before starting with the next
symbol. Some symbols (it does not matter which) should be designated to occur
$\big\lfloor2(a-c)(b-c)/|S|\big\rfloor$ times, while the others occur
$\big\lceil2(a-c)(b-c)/|S|\big\rceil$ times. Note that $|S|\ge b-c$ and hence
$\big\lceil2(a-c)(b-c)/|S|\big\rceil\le2(a-c)$, which means that our
construction will not violate the latin property.
\end{proof}

The following corollary gives a sufficient (though not in general necessary)
condition, and it is easier to apply than \tref{t:overlapsubsq}.

\begin{corollary}\label{c:subsquares}
Suppose $0<c<a\le b<n$ are integers. If $n-2b\ge a-2c$ and $a\ge \frac{5}{2}c$
then there is a latin square of order $n$ containing subsquares of order $a$
and $b$ whose intersection is a subsquare of order $c$.
\end{corollary}

\begin{proof}
The two conditions $n-2b\ge a-2c$ and $a\ge \frac{5}{2}c$ imply
\eref{e:easycond}, so it remains to show \eref{e:hardcond}.

If $n\ge 2b+c$ then $(n-2a)(n-2b)\ge c^2$ and \eref{e:hardcond} follows. If
$n\le 2a+2b-4c$ then $(n-2a-2b+3c)^2\ge c^2$, so the right hand side of
\eref{e:hardcond} is negative, and \eref{e:hardcond} follows once again.

Otherwise we have $n<2b+c$ and $n>2a+2b-4c$, which gives $5c>2a$, a
contradiction.
\end{proof}

Next, we examine the case of a latin square with two subsquares that do not
overlap.

We say that two subsquares \emph{share a row} if there is some row
that intersects both subsquares. Sharing a column or symbol is defined
analogously.  By the operation known as {\em conjugacy} or {\em
parastrophy} (a permutation of the roles of rows, columns and
symbols in a latin square) we may assume that if the subsquares share
anything, they share a row.  If two disjoint subsquares share a row,
it immediately follows that they do not share a column or symbol.

%???
%The assumption that the two squares do not share a symbol is therefore
%needed only in the situation when the two disjoint subsquares do not
%share a row or column.

\begin{lemma}\label{l:sharesome}
Suppose $0\le c<a\le b<n$ are integers. In order for there to exist a latin
square of order $n$ containing two subsquares of order $a$ and $b$ that share
exactly $c$ rows and do not share any columns or symbols, it is necessary and
sufficient that $n\ge a+2b$.
\end{lemma}

\begin{proof}Up to permutation of the rows and columns, our latin square, if
it exists, looks like
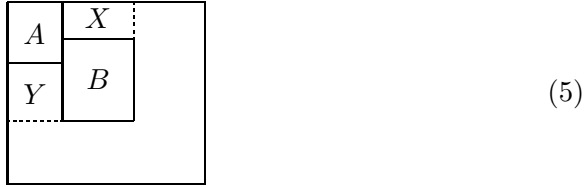
\begin{equation}%\label{e:formofL}
\begin{tabular}{|c|c;{1pt/1.5pt}c|}
\hline
\multicolumn{1}{|l|}{\smash{\lower6pt\hbox{$A$}}}&$X$&\\
\cline{2-2}
&\multicolumn{1}{r|}{}&\\[-1ex]
\cline{1-1}
$Y$&\multicolumn{1}{r|}{\smash{\raise6pt\hbox{$\;B\;$}}\vrule height 16pt depth 6pt width 0pt}
&\vrule height 0pt depth 0pt width 15pt\\
\cdashline{1-1}[1pt/1.5pt]\cline{2-2}
\multicolumn{3}{|c|}{\vrule height 20pt width 0pt}\\
\hline
\end{tabular}\,
\end{equation}
where $X$ and $Y$ are (non-empty) submatrices of dimensions $(a-c)\times b$ and
$(b-c)\times a$ respectively. The symbols in $X\cup Y$ must be distinct from
those in $A\cup B$, and by assumption, the symbols in subsquare $A$ are
distinct from the symbols in subsquare $B$. Each row of $X$ must contain
$b$ symbols that are different from the $a+b$ symbols in $A\cup B$, which
demonstrates the necessity of the condition $n\ge 2b+a$.

To prove the sufficiency of the condition $n\ge 2b+a$, we apply \tref{t:ryser}
to $R=A\cup B\cup X\cup Y$.  Symbols in $A$ each occur $a$ times in $R$, so we
need $a\ge a+b-c+a+b-n$, that is, $n\ge a+2b-c$. Similarly, considering the
symbols in $B$ leads to $n\ge2a+b-c$. Both of these inequalities hold if
$n\ge2b+a$, given that $a\le b$.

We fill in the symbols in $X\cup Y$ in the same way that we did in
\tref{t:overlapsubsq}. Each symbol in $X\cup Y$ will occur at most
\begin{align*}
\bigg\lceil\frac{(a-c)b+(b-c)a}{n-a-b}\bigg\rceil
&\le\bigg\lceil\frac{(a-c)b+(b-c)a}{b}\bigg\rceil\\
&= a-c+\bigg\lceil\frac{(b-c)a}{b}\bigg\rceil\\
&\le a-c+\min\{b-c,a\}
\end{align*}
times, which means that we will not allocate the same symbol to two different
cells in the same row or column.

Moreover, the symbols in $X\cup Y$ will satisfy Ryser's condition, given that
$(a-c)b+(b-c)a\ge(n-a-b)(2a+2b-c-n)$. This last condition is algebraically
equivalent to $(n-2a-b)(n-a-2b+c)+(a-c)b\ge0$, which is obviously true.
\end{proof}

\lref{l:sharesome} did not cover the case when $a=c$. It needs to be treated
separately:

\begin{lemma}\label{l:shareall}
Suppose $0<a\le b<n$ are integers. In order for there to exist a latin square
of order $n$ containing two subsquares of order $a$ and $b$ that share exactly
$a$ rows and do not share any columns or symbols, it is necessary and
sufficient that either $a=b=n/2$ or $n\ge\max\{2a+b,2b\}$.
\end{lemma}

\begin{proof}First suppose that $n>a+b$. Then, 
up to permutation of the rows and columns, our latin square looks like this:
\begin{equation}\label{e:formofL'}
\begin{tabular}{|c|c|c|}
\hline
{\smash{\lower8pt\hbox{$\;\;B\;\;$}}}&$A$\vrule depth 7pt height 14pt width 0pt&$\;X\;$\\
\cline{2-2}\cdashline{3-3}[1pt/1.5pt]&\multicolumn{2}{c|}{$Y$\vrule depth 0pt height 12pt width 0pt}\\
\cline{1-1}\cdashline{2-3}[1pt/1.5pt]
\multicolumn{3}{|c|}{}\\
\multicolumn{3}{|c|}{}\\
\multicolumn{3}{|c|}{}\\
\hline
\end{tabular}
\end{equation}
where as usual, $A$ and $B$ are the subsquares of order $a$ and $b$.

Once we are given $A$, \tref{t:ryser} shows that to be able to fill in $A\cup
X\cup Y$ it is necessary and sufficient that $0\ge 2a-(n-b)$ and $b\le n-b$.
The remainder of the latin square can then always be completed. Hence, for the
latin square to exist when $n>a+b$, it is necessary and sufficient that
$n\ge\max\{2a+b,2b\}$.

That leaves the case $n=a+b$. In that case, the submatrix $X$ in
\eref{e:formofL'} has no columns. Also, there are no symbols available to fill
the submatrix $Y$, so we are forced to make $a=b$. With that condition, it is
trivial to complete a latin square of order $n=2a=2b$.
\end{proof}

%\begin{problem} Characterize the parameters $0<a\le b\le n$ for which there
%exists a latin square of order $n$ with two disjoint subsquares of orders $a$
%and $b$, respectively, that share a symbol.
%\end{problem}

We have now considered all possible ways that a latin square might
contain subsquares of order $a$ and $b$. \tref{t:overlapsubsq} (and
its preamble) handles the case when the subsquares intersect. Disjoint
subsquares that share any row, column or symbol are covered (up to
conjugacy/parastrophy) by \lref{l:sharesome} and
\lref{l:shareall}. Subsquares that do not share any row, column or
symbol are covered explicitly by \lref{l:sharesome} (the same result
would be obtained by taking $c=0$ in \tref{t:overlapsubsq}, although
for simplicity we did not allow that case in our phrasing of the
theorem). Putting these results together, we get:

\begin{theorem}
Suppose $1<a \leq b <n$ are integers.  In order for there to exist a
latin square of order $n$ containing two distinct subsquares of order
$a$ and $b$, the following set of conditions is necessary and sufficient:
\begin{itemize}
\item[(a)]  $n \geq 2b$,
\item[(b)]  if $n=2b+1$ and $a$ is even then $a \leq \frac{2}{3}(b+1)$ and
\item[(c)]  if $n=2b$ and $a$ is odd then $a=b$ or $a \leq b/2$.
\end{itemize}
\end{theorem}

\begin{proof} As usual, $A$ and $B$ denote subsquares of order $a$ and $b$ in
a latin square $L$ of order $n$.

Condition (a) is clearly necessary.
If $n\ge2b+a$ then we can use \lref{l:sharesome}, so we may assume that
$2b\le n<2b+a$.

Suppose first that $n\equiv a\mod2$. Take $c=(a+2b-n)/2$ and note that
\eref{e:easycond} holds. Moreover,
\[(n-2a)(n-2b)\ge0\ge a(2b-n)=c^2-(n-2a-2b+3c)^2\]
so \eref{e:hardcond} holds as well and we are done. Henceforth we may
assume that $n\not\equiv a\mod2$. At this point we split into three cases.

\def\case{\medskip\noindent{\bf Case: \ }}

\case $n\ge2b+2$

Take $c=(a+2b-n+1)/2$ and note that \eref{e:easycond} holds. Moreover,
\[(n-2a)(n-2b)\ge0\ge (a-1)(2b-n+2)=c^2-(n-2a-2b+3c)^2\]
so \eref{e:hardcond} holds as well and we are done.

\case $n=2b+1$

Then $a$ is even since $n\not\equiv a\mod2$. If $A$ and $B$ intersect
in a smaller subsquare then \eref{e:easycond} implies that we must
have $a=2c$. In turn, \eref{e:hardcond} tells us that
$2b+1\ge3a-1$. Hence if $a\leq\frac23(b+1)$ then
\tref{t:overlapsubsq} shows that $L$ exists. Otherwise
$a>\frac23(b+1)>\frac12b$, so the only hope of constructing $L$ 
is that $A$ and $B$ do not
overlap. However, we have $n=2b+1<b+\frac32a<b+2a\le2b+a$, which
breaches the necessary conditions in both \lref{l:sharesome} and
\lref{l:shareall}. So $A$ and $B$ cannot be disjoint.

\case $n=2b$

A latin square of order $n$ containing a subsquare of order $n/2$ must
be composed of four disjoint subsquares of order $n/2$. Thus by
changing our choice of $B$ if necessary, we can assume that $a=b$ or
that $A$ and $B$ intersect.  Since $n\not\equiv a\mod2$ we find that
$a$ is odd, but this makes \eref{e:easycond} impossible to satisfy. So
$A$ and $B$ cannot intersect unless $A$ is wholly inside $B$, in which
case $a=b$ or $a\le b/2$ (and both options are achievable).
\end{proof}

\section{A bound on the number of subsquares}

The following simple bound on the number of subsquares in a latin
square generalises results by Heinrich and Wallis \cite{HW} and van
Rees \cite{vR} who proved the $m=2$ and $m=3$ cases, respectively.

\begin{theorem}\label{t:subsqbnd}
Let $m,n$ be positive integers, and define $h=\big\lceil(m+1)/2\big\rceil$.
No latin square of order $n$ may have more than
\[\frac{n{n \choose h}}{m{m\choose h}}\]
subsquares of order $m$.
\end{theorem}

\begin{proof}
Suppose $L$ is a latin square of order $n$, and consider a set $H$ of
$h$ entries from within one row of $L$.  Since $h>m/2$, we know from
\eref{e:easycond} that at most one $m\times m$ subsquare of $L$ can
contain $H$. There are $n{n\choose h}$ possible choices for $H$, and
each $m\times m$ subsquare contains $m{m\choose h}$ of them.
\end{proof}

%% Original proof
%\begin{proof}
%Suppose $L$ is a latin square of order $n$, and consider a set $H$ of
%$h$ rows of $L$.  Suppose $S$ is an $m\times m$ subsquare of $L$ that
%uses all the rows in $H$. The intersection of $S$ and $H$ will be an
%$h\times m$ latin rectangle $R_S$. Since $h>m/2$, we know from
%\eref{e:easycond} that no $m\times m$ subsquare of $L$ other than $S$
%can contain an entire column of $R_S$. It follows that $H$ cannot
%contain $R_S$ for more than $n/m$ different choices of $S$. There are
%$n\choose h$ possible choices for $H$, and each $m\times m$ subsquare
%will be counted in $m\choose h$ of them.
%\end{proof}

By refining the idea of \tref{t:subsqbnd} we can prove a better bound
when $m$ is large. All asymptotics in the remainder of this section
are for $n\rightarrow\infty$ with other quantities fixed.

\begin{theorem}\label{t:bttrbnd}
Let $m,n,t$ be positive integers with $n\ge m\ge t$.
Define $\psi(m,t)$ by
\[
\psi(m,t)=
\begin{cases}
\big\lfloor m/(2 t) \big\rfloor+1 & \mbox{if $m$ is even},\\
\big\lceil \frac{1}{2} \lfloor m/t\rfloor \big\rceil & \mbox{if $m$ is odd}.
\end{cases}
\]
No latin square of order $n$ can have more than $O(n^{\psi(m,t)+t})$
subsquares of order $m$.
\end{theorem}

\begin{proof}
Let $L$ be a latin square of order $n$. Define a {\em block} to be
a latin subrectangle of $L$ with exactly $t$ rows, that is minimal in
the sense that it does not contain any smaller such latin rectangle.

Suppose that we choose up to $\psi(m,t)$ blocks that lie in the same
$t$ rows of $L$.  We claim that for every subsquare $S$ of order $m$
in $L$ there is (at least) one such choice of blocks that lies in $S$ and not
in any other subsquare of order $m$. It will then follow that the number
of $m\times m$ subsquares is not more than $O(n^{\psi(m,t)+t})$ since there are
$O(n^t)$ ways to choose $t$ rows and $O(n^{\psi(m,t)})$ ways to choose up
to $\psi(m,t)$ blocks in those rows.

To prove our claim, we choose blocks from the first $t$ rows of $S$
until the chosen blocks cover strictly more than $m/2$ columns.  By
\eref{e:easycond} this will guarantee that no other subsquare of order
$m$ contains these blocks.  At each step we select a block that is at
least as large as any of the remaining blocks.  This ensures that we
need no more than $\psi(m,t)$ blocks, as the following argument shows.

There are at most $\lfloor m/t\rfloor$ blocks within the first $t$
rows of $S$, since each block uses at least $t$ columns.  If $m$ is
odd, we need only choose at least half of them and we are done.  So
suppose $m$ is even. If $m/2$ is divisible by $t$ then
$t\psi(m,t)=\frac12 m+t>\frac12 m$ and otherwise $t\psi(m,t)=t\lceil
m/(2t)\rceil>\frac12 m$. Thus, as claimed, there is no case where we need
more than $\psi(m,t)$ blocks to determine $S$.
\end{proof}

\begin{corollary}\label{c:sqrtm}
No latin square of order $n$ can have more than
$O(n^{\sqrt{2m}+2})$ subsquares of order $m$.
\end{corollary}

\begin{proof}
Taking $t=\big\lceil\sqrt{m/2}\,\big\rceil$, we will show that
$\psi(m,t)+t \leq \sqrt{2m}+2$.  If $m$ is even then
\begin{align*}
\psi(m,t)+t
&=\bigg\lfloor\frac{m}{2\lceil \sqrt{m/2}\,\rceil} \bigg\rfloor+1+\big\lceil\sqrt{m/2}\,\big\rceil\\
&\le \big\lfloor \sqrt{m/2} \big\rfloor+1+\sqrt{m/2}+1.
%\\&\le2\big\lceil\sqrt{m/2}\,\big\rceil.
\end{align*}
Similarly, if $m$ is odd,
\[%\begin{align*}
\psi(m,t)+t
=\bigg\lceil \frac{1}{2} \bigg\lfloor \frac{m}{\lceil \sqrt{m/2}\,\rceil} \bigg\rfloor \bigg\rceil+\big\lceil\sqrt{m/2}\,\big\rceil
\le2\big\lceil\sqrt{m/2}\,\big\rceil.
\]
In either case, $\psi(m,t)+t\le2(\sqrt{m/2}+1)$, as required.
\end{proof}

The bound in \tref{t:subsqbnd} is obviously achieved when $m=n$ and is
known \cite{HW,vR} to be achieved for infinitely many $n$ when
$m\in\{2,3\}$.  It is also not hard to show that the elementary
abelian 2-groups achieve the bound in \tref{t:subsqbnd} for $m=4$.  On
the other hand, \cref{c:sqrtm} shows that the bound in
\tref{t:subsqbnd} can only be achieved for finitely many $n$ when
$m>9$. In fact, a more careful analysis using \tref{t:bttrbnd} reveals
that, for $m>4$, the number of subsquares of order $m$ is
$o(n^{1+h})$. Hence \tref{t:subsqbnd} is not
best possible except for $m\le4$.

See \cite{MW} for some results on how many subsquares are `typical'
in a random latin square. In that paper it is conjectured that
the proportion of latin squares of order $n$ with a subsquare of order
greater than 3 tends to zero as $n\rightarrow\infty$.

\section{Full products}

A \emph{quasigroup} is a groupoid $(Q,\cdot)$ such that the equations $ax=b$
and $ya=b$ have unique solutions $x$, $y\in Q$ for every $a$, $b\in Q$. A
\emph{loop} is a quasigroup with a neutral element.
Latin squares are precisely multiplication tables of finite quasigroups.
Multiplication tables of finite loops correspond to normalized latin squares.

For a subset $S$ of a loop $Q$, denote by $P(S)$ the set of all elements of $Q$
that are obtained as products of all elements of $S$ with each element of $S$
being used precisely once. We refer to elements of $P(Q)$ as \emph{full
products} of $Q$. Full products play an important role in a recent
non-associative interpretation for the Hall-Paige conjecture \cite{pula}.

For a loop $Q$, let $Q'$ be the \emph{derived subloop}, that is, the least
normal subloop $H$ of $Q$ such that $Q/H$ is an abelian group.

It is not difficult to show that $P(Q)$ is contained in a coset of
$Q'$. The D\'enes-Hermann Theorem \cite{DH} states that if $Q$ is a
group then $P(Q)$ is equal to a coset of $Q'$; more precisely, either
$P(Q)=Q'$ or $P(Q)=xQ'$ where $x^2\in Q'$.

It is not true for a general loop $Q$ that $P(Q)$ is a coset of $Q'$, but the
only known counterexamples are of order $5$. For instance, the loop
\begin{displaymath}
    \begin{array}{c|ccccc}
        Q_1&1&2&3&4&5\\
        \hline
        1&1&2&3&4&5\\
        2&2&1&5&3&4\\
        3&3&4&1&5&2\\
        4&4&5&2&1&3\\
        5&5&3&4&2&1
    \end{array}
\end{displaymath}
satisfies $P(Q_1)=\{2,3,4,5\}$.

Note that $|P(Q)|$ is not an isotopy invariant, since the loop
\begin{equation}\label{e:full5}
    \begin{array}{c|ccccc}
        Q_2&1&2&3&4&5\\
        \hline
        1&1&2&3&4&5\\
        2&2&1&4&5&3\\
        3&3&4&5&2&1\\
        4&4&5&1&3&2\\
        5&5&3&2&1&4
    \end{array}
\end{equation}
is isotopic to $Q_1$ but satisfies $P(Q_2)=Q_2$.

While we can certainly have $P(Q)<Q$ in an arbitrarily large loop (an abelian
group $Q$ will do), it is to be expected that a sufficiently large loop $Q$
chosen at random will satisfy $Q'=Q$ and, in fact, $P(Q)=Q$. However, we are
not aware of any argument that would show $P(Q)=Q$ for a `typical' sufficiently
large loop $Q$. As an application of our previous results we show here that for
every $n\ge 5$ there is a loop $Q$ of order $n$ satisfying $P(Q)=Q$.

\begin{lemma}\label{l:tripling}
Let $Q$ be a loop of order $m\ge 5$ such that $P(Q)=Q$. Then for every $3m-2\le
n\le 4m-3$ there is a loop $H$ of order $n$ satisfying $P(H)=H$.
\end{lemma}
\begin{proof}
Let $\overline{Q}$ be an isomorphic copy of $Q$. Set $a=b=m$ and $c=1$ in
\cref{c:subsquares} to see that for every $n\ge 3m-2$ there is a loop $H$ of
order $n$ that contains $Q$ and $\overline{Q}$ as subloops and such that
$Q\cap\overline{Q}=1$, the neutral element of $H$.

Assume further that $n\le 4m-3$ and let $x\in H\setminus (Q\cup \overline{Q})$.
By Ryser's condition, $x$ appears in $(Q\cup \overline{Q})\times (Q\cup
\overline{Q})$ at least $2(2m-1) - n \ge 1$ times, and hence either $x\in
Q\times \overline{Q}$ or $x\in \overline{Q}\times Q$.

We are going to show that $P(Q\cup \overline{Q}) = H$. This will imply that
$P(H)=H$ since the cardinality of $P(Q\cup \overline{Q})$ cannot decrease upon
multiplying if by the elements of $H\setminus(Q\cup\overline{Q})$ in any way.

Let $x\in H$. If $x\in Q$, we see that $x\in P(Q\cup \overline{Q})$ since
$P(Q)=Q$ and $1\in P(\overline{Q})$. Similarly, $x\in P(Q\cup\overline{Q})$ for
every $x\in \overline{Q}$. Suppose that $x\in H\setminus(Q\cup\overline{Q})$.
By the argument above, we know that either $x = y\overline{y}$ or
$x=\overline{y}y$ for some $y\in Q$, $\overline{y}\in \overline{Q}$. This means
that $x\in P(Q\cup\overline{Q})$, as $P(Q)=Q$ and
$P(\overline{Q})=\overline{Q}$.
\end{proof}

\begin{theorem} There is a loop $Q$ of order $n$ satisfying $P(Q)=Q$ if and
only if $n=1$ or $n\ge 5$.
\end{theorem}
\begin{proof}
The statement is true for $n=1$. Assume that $1<n<5$. Then $Q$ is an
abelian group, $Q'<Q$, and thus $P(Q)$ is a proper subset of $Q'$. In
\eref{e:full5} we gave a loop $Q_2$ of order 5 where $Q_2=P(Q_2)$. It
is easy to check by computer that for every $5< n\le 12$ there is a
loop $Q$ of order $n$ satisfying $P(Q)=Q$.  (In fact we did not find
any example where $P(Q)<Q$, except those with $Q'<Q$.)  By
\lref{l:tripling}, the theorem is also true for all $n$ in the
intervals $[3\cdot 5-2,4\cdot 5-3] = [13,17]$, $[3\cdot 6-2,4\cdot 6-3] 
= [16,21]$, $[3\cdot 7-2,4\cdot 7-3] = [19,25]$, and so on,
obviously accounting for every $n\ge 5$.
\end{proof}

\section*{Acknowledgement}

The authors are grateful to Doug Stones for useful feedback
on a draft of this paper.

%%%  Squashing the bibliography together a bit
  \let\oldthebibliography=\thebibliography
  \let\endoldthebibliography=\endthebibliography
  \renewenvironment{thebibliography}[1]{%
    \begin{oldthebibliography}{#1}%
      \setlength{\parskip}{0.4ex plus 0.1ex minus 0.1ex}%
      \setlength{\itemsep}{0.4ex plus 0.1ex minus 0.1ex}%
  }%
  {%
    \end{oldthebibliography}%
  }


\begin{thebibliography}{99}

\bibitem{DH} J.~D\'enes and P.~Hermann, 
On the product of all elements in a finite group,
{\it Ann. Discrete Math.} {\bf15} (1982), 105--109.

\bibitem{HW} K.~Heinrich and W.\,D.~Wallis, 
The maximum number of intercalates in a latin square, 
Lecture Notes in Math.\ {\bf884} (1981), 221-233.

\bibitem{MW}
B.\,D.~McKay and I.\,M.~Wanless, Most latin squares have many subsquares,
{\it J.\ Combin.\ Theory Ser.\ A\/} {\bf86} (1999), 323--347.

\bibitem{pula} K.~Pula,
Products of all elements in a loop and a framework for non-associative
analogues of the Hall-Paige conjecture,
{\it Electron. J. Combin.} {\bf16} (2009), R57.

\bibitem{ryser} H.\,J.~Ryser, A combinatorial theorem with an application to
    latin rectangles, {\it Proc. Amer. Math. Soc.} {\bf2} (1951), 550--552.

\bibitem{vR} G.\,H.\,J.~van Rees, Subsquares and transversals in latin
squares, {\em Ars Combin.\/} {\bf29}B (1990), 193--204.

\end{thebibliography}
\end{document}